\documentclass[10pt]{article}
\usepackage{mathrsfs}
\usepackage{amsfonts}
\usepackage{amsthm,amsmath,amssymb,anysize}
\newtheorem{lemma}{Lemma}[section]
\newtheorem{theorem}[lemma]{Theorem}

\newtheorem{proposition}[lemma]{Proposition}

\usepackage[numbers,sort&compress]{natbib}
\marginsize{3.6cm}{3.6cm}{3.6cm}{3cm}
\numberwithin{equation}{section}

\title{\textsf{ Super-biderivations of Cartan type Lie superalgebras}}
\author{\textsc{Jixia Yuan,$^{1,2,}$}\footnote{Supported by  the National Natural Science Foundation of China (11601135), Natural Science Foundation of Heilongjiang Province of China (QC2017002) and Project funded by China Postdoctoral Science Foundation}\;\;\textsc{Liangyun Chen$^{1,}$}\footnote{Correspondence: chenly640@nenu.edu.cn. (L. Chen);  Supported by the National Natural Science Foundation
  of China (11771069, NSF of Jilin province (No. 20170101048JC) and the project of jilin province department of education (No. JJKH20180005K))}\;\;  \textsc{ and Yan Cao$^{3,}$}\footnote{Supported by  the National Natural Science Foundation of China (11801121), Natural Science Foundation of Heilongjiang Province of China (QC2018006) } \\
  \\
  \textit{1. School of Mathematical and Statistics},
  \textit{Northeast Normal University} \\
  \textit{Changchun 130024, China}\\
\\
  \ \ \textit{2. School of Mathematical Sciences},
  \textit{Heilongjiang University} \\
  \textit{Harbin 150080, China}\\
  \\
  \ \ \textit{3.  Department of Mathematics},
  \textit{Harbin University of Science and Technology} \\
  \textit{Harbin 150080, China}
  }
\date{ }
\begin{document}
\maketitle
\begin{quotation}
\noindent\textbf{Abstract.} In this paper, we characterize the  super-biderivations of  Cartan type Lie superalgebras over the complex field $\mathbb{C}$. Furthermore, we prove that all super-biderivations of Cartan type simple Lie superalgebras  are inner super-biderivations.
\\



  \end{quotation}

  \setcounter{section}{0}
\section{Introduction}

As is well known,  all  the finite dimensional simple Lie  superalgebras over $\mathbb{C}$ consist of classical   Lie superalgebras  and Cartan type Lie superalgebras \cite{k1}. However, four families of $\mathbb{Z}$-graded Lie superalgebras of Cartan-type are, the
Witt type, the special type and the Hamiltonian type.  These Lie superalgebras are subalgebras of
the full superderivation algebras of the  exterior
superalgebras. The structural theory of these superalgebras have been playing a key role in the theory of  Lie superalgebras.
Derivations and generalized derivations are very important subjects in the
research of both algebras and their generalizations.

The concept of biderivations were introduced and studied    in \cite{Bre1995}, and the authors showed that all biderivations on
commutative prime rings are inner biderivations and determined the biderivations of semiprime rings. In recent years, biderivations have aroused the interest of a great many authors, see  \cite{Bre1995,Du2013,Gho2013}.   The notion of biderivations of Lie algebras was introduced in \cite{wyc}.  The  biderivations of Lie  algebras have been sufficiently studied. In \cite{c,c2017,wy}, the authors
proved that all anti-symmetric biderivations on the Schr$\ddot{o}$dinger-Virasoro algebra, the
simple generalized Witt algebra and the infinite-dimensional Galilean conformal algebra
are inner, respectively. In \cite{Hanw}, the
authors determined all the skew-symmetric biderivations of W(a;b) and found that there
exist non-inner biderivations. In \cite{t1,t2}, the authors characterized the biderivations without
the anti-symmetric condition of the finite dimensional complex simple Lie algebra and
some W-algebras, meanwhile, presented some classes of non-inner biderivations. In 2015,  Xu
and  Wang  generalized the notion of biderivations of Lie  algebra to the
super case in \cite{xw}, and the authors mainly discussed some properties of super-biderivations on Heisenberg superalgebras. In \cite{fd}, the authors proved that all super-skewsymmetric super-biderivations on the centerless super-Virasoro algebras are inner super-biderivations. In \cite{xwh}, the authors proved  that any super-skewsymmetric super-biderivation of super-Virasoro  algebras was inner. In \cite{yt},  the authors  proved that all super-biderivations of classical simple Lie superalgebras  were inner super-biderivations and finded a class of a non-inner and non-skewsymmetric
 super-biderivations of general linear Lie superalgebra.  In \cite{lty}, the authors
 determined all the  biderivations of $W(2,2)$ and found that there
existed non-inner biderivations.

In this paper, we are interested in determining all super-biderivations of Cartan type Lie superalgebras over $\mathbb{C}$ without the skew-symmetric condition.   The main result in this paper is a complete characterization of the super-biderivations of  Cartan type Lie superalgebras.
The paper is organized as follows. In Section 2, we recall some necessary concepts and notations. In   Section 3, we  establish  several  lemmas, which will be used to characterize the super-biderivations of classical simple Lie superalgebras. In Section 4, we determine all super-biderivations of Cartan type Lie superalgebras $W(n), S(n), \widetilde{S}(n)$ without the super-skewsymmetric condition.  In Section 5,  we determine all super-biderivations of Cartan type Lie superalgebras $H(n)$ without the super-skewsymmetric condition.

\section{Preliminaries}
Throughout $\mathbb{C}$ is the  field of complex numbers, $\mathbb{N}$ the
set of nonnegative integers and
$\mathbb{Z}_2= \{\bar{0},\bar{1}\}$   the additive group of two
elements.  For a  vector superspace $V=V_{\bar{0}}\oplus
V_{\bar{1}}, $ we write $|x|$ for the
parity of    $x\in V_{\alpha},$ where $\alpha\in \mathbb{Z}_2$. Once the symbol $|x|$ appears
in this paper, it
 will imply that $x$ is  a  $\mathbb{Z}_2$-homogeneous element. We also adopt the following notation: For  a proposition $P$, put
 $\delta _{P}=1$ if $P$ is true and $\delta _{P}=0$ otherwise.

 \subsection{Lie algebras, derivation and biderivation}
Let us  recall some definitions relative to Lie algebras \cite{h}. A Lie algebra is a vector space $L$ with a   bilinear   mapping
$[\cdot,\cdot]: L\times L\longrightarrow L$   satisfying the following axioms:
\begin{eqnarray*}
&& [x,y]=-[y,x],\\
&&[x,[y,z]]=[[x,y],z]+[y,[x,z]]
\end{eqnarray*}
for all $x, y, z\in L.$

We call a  bilinear map $f:L\times L \longrightarrow L$ a  biderivation of $L$ if it satisfies the
following two equations:
\begin{eqnarray*}
&& f([x,y], z)=[x, f(y,z)]+[f(x,z), y],\\
&& f(x, [y, z])=[f(x,y), z]+[y, f(x,z)]
\end{eqnarray*}
for all $x, y$, $z \in L.$
If the map $f_{\lambda}:L\times L \longrightarrow L,$ defined by $f_{\lambda}(x,y)=\lambda[x, y]$ for all  $x, y\in L,$ where $\lambda\in \mathbb{C},$ then it is easy to check that $f_{\lambda}$ is a biderivation of $L$. We call this class of biderivations by inner biderivations.

\subsection{Lie superalgebras, super-derivation and  super-biderivation}
Let us  recall some definitions relative to Lie superalgebras \cite{k1}. A Lie superalgebra is a vector superspace $L=L_{\bar{0}}\oplus
L_{\bar{1}}$ with an even bilinear   mapping
$[\cdot,\cdot]: L\times L\longrightarrow L$   satisfying the following axioms:
\begin{eqnarray*}
&& [x,y]=-(-1)^{|x||y|}[y,x],\\
&&[x,[y,z]]=[[x,y],z]+(-1)^{|x||y|}[y,[x,z]]
\end{eqnarray*}
for all $x, y, z\in L.$

We call a  linear map $D:  L \longrightarrow L$ a  super-derivation of Lie superalgebra $L$ if it satisfies the
following   equation:
\begin{eqnarray*}\label{yuantange1}
&& D([x,y])=[D(x), y]+(-1)^{|D||x|}[x, D(y)]
\end{eqnarray*}
for all $x, y \in L.$ Write $\mathrm{Der}_{\bar{0}}L$ (resp. $\mathrm{Der}_{\bar{1}}L$) for the set of all super-derivations
of $\mathbb{Z}_{2}$-homogeneous $\bar{0}$ (resp.  $\bar{1}$) of $L$.  Denote
$$\mathrm{Der}L=\mathrm{Der}_{\bar{0}}L\oplus \mathrm{Der}_{\bar{1}}L.$$

We call a  bilinear map $f:L\times L \longrightarrow L$ a super-biderivation of $L$ if it satisfies the
following two equations:
\begin{eqnarray}
&& f([x,y], z)=(-1)^{|x||f|}[x, f(y,z)]+(-1)^{|y||z|}[f(x,z), y],\label{e1}\\\nonumber
&& f(x, [y, z])=[f(x,y), z]+(-1)^{(|f|+|x|)|y|}[y, f(x,z)] \label{e3}
\end{eqnarray}
for all $x, y$, $z \in L.$ Note that our definition of super-biderivation  does not assume particular super-skewsymmetry which is different from \cite{fd}.
If the map $f_{\lambda}:L\times L \longrightarrow L$ is defined by $f_{\lambda}(x,y)=\lambda[x, y]$ for all  $x, y\in L,$ where $\lambda\in \mathbb{C},$ then it is easy to check that $f_{\lambda}$ is a super-biderivation of $L$. This class of super-biderivations is said to be inner. Denote by $\mathrm{IBDer}(L)$ the set of all inner super-biderivations.

A super-biderivation $f$ of $\mathbb{Z}_{2}$-homogeneous $\gamma$ of $L$ is a super-biderivation such
that $f(L_{\alpha}, L_{\beta})\subseteq L_{\alpha+\beta+\gamma}$ for any $\alpha, \beta\in \mathbb{Z}_{2}$. Write $\mathrm{BDer}_{\bar{0}}(L)$ (resp. $\mathrm{BDer}_{\bar{1}}(L)$) for the set of all super-biderivations
of $\mathbb{Z}_{2}$-homogeneous $\bar{0}$ (resp.  $\bar{1}$) of $L$.  Denote
$$\mathrm{BDer}(L)=\mathrm{BDer}_{\bar{0}}(L)\oplus \mathrm{BDer}_{\bar{1}}(L).$$

\subsection{Cartan type Lie superalgebras}

Let $n\geq 4$ be   an   integer and $\Lambda(n)$ be the exterior algebra in $n$ indeterminates $x_{1}, x_{2},\ldots, x_{n}$ with $\mathbb{Z}_{2}$-grading structure given by $|x_{i}|=\bar{1}.$   One may define a $\mathbb{Z}$-grading on $\Lambda(n)$ by letting $\mathrm{deg}x_{i}=1,$ where $1 \leq i\leq n.$ Write $n=2r$ or $n=2r+1,$ where $r\in \mathbb{N}.$ Put
 $\left[\frac{n}{2}\right]=r.$
 Cartan type Lie superalgebras consist of
four series of  simple Lie superalgebras contained in the full superderivation algebras of $\Lambda(n)$:
\begin{eqnarray*}
&&W(n)=\left\{\sum_{i=1}^{n}f_{i}\partial_{i}\mid  f_{i}\in \Lambda(n)\right\},
\\&&
S(n)=\left\{\sum_{i=1}^{n}f_{i}\partial_{i}\mid  f_{i}\in \Lambda(n), \sum_{i=1}^{n}\partial_{i}(f_{i})=0\right\},
\\&&
\widetilde{S}(n)=\left\{(1-x_{1}x_{2}\cdots x_{n})\sum_{i=1}^{n}f_{i}\partial_{i}\mid  f_{i}\in \Lambda(n), \sum_{i=1}^{n}\partial_{i}(f_{i})=0\right\}\; \mbox{(}n \;\mbox{is an even integer}\mbox{)},
\\&&
H(n)=\left\{\mathrm{D_{H}}(f)\mid  f\in \oplus_{i=0}^{n-1}\Lambda(n)_{i}\right\} \;\mbox{(}n> 4\mbox{)},
\end{eqnarray*}
where
$$\mathrm{D_{H}}(f)=(-1)^{|f|}\sum_{i=1}^{n}\partial_{i}(f)\partial_{i'},$$
\[ i'=\left\{
 \begin{array}{lll}
i+r,&\mbox{if  $1 \leq i\leq \left[\frac{n}{2}\right],$ }\\
i-r,&\mbox{if $ \left[\frac{n}{2}\right]< i\leq 2\left[\frac{n}{2}\right],$}\\
i,&\mbox{otherwise}.\\
\end {array}
\right.
\]
A simple computation shows
\begin{eqnarray}\label{hee1.2}
[\mathrm{D_{H}}(f),\mathrm{D_{H}}(g)]=\mathrm{D_{H}}(\mathrm{D_{H}}(f)(g))\;\mbox{for}\;
   f,g\in \Lambda(n).
\end{eqnarray}
 One may define a $\mathbb{Z}$-grading on $W(n)$ by letting $\mathrm{deg}x_{i}=1=-\mathrm{deg}\partial_{i},$ where $1 \leq i\leq n.$ Thus $W(n)$   becomes a $\mathbb{Z}$-graded Lie superalgebra of 1 depth: $W(n)=\bigoplus_{j=-1}^{\xi_{W}}W(n)_{j},$ where $\xi_{W}=n-1.$
Suppose $L=S(n)$ or $H(n)$. Then $L$ is a  $\mathbb{Z}$-graded subalgebra  of $W(n)$. The $\mathbb{Z}$-grading is defined as follows:
  $L=\bigoplus_{j=-1}^{\xi_{L}}L_{j},$
  where $L_{j}=L\cap W(n)_{j}$ and
 \[
\xi_{L}=\left\{
 \begin{array}{ll}
n-2,&\mbox{if  $L=S(n),$ }\\
n-3,&\mbox{if $H(n)$.}
\end {array}
\right.
\]
Put $$\widetilde{S}(n)_{-1}=\mathrm{span}_{\mathbb{C}}\left\{(1-x_{1}x_{2}\cdots x_{n})\partial_{i}\mid  1\leq i \leq n\right\},\; \widetilde{S}(n)_{i}=S(n)_{i} \;\mbox{for\;} i>-1.$$
Then $\widetilde{S}(n)$ becomes a $\mathbb{Z}_{n}$-graded Lie superalgebra:
$
\widetilde{S}(n)=\bigoplus_{i=-1}^{\xi_{\widetilde{S}}}\widetilde{S}(n)_{i},
 $
 where $\xi_{\widetilde{S}}=n-2.$
The $0$-degree components of these superalgebras are classical Lie algebras:
$$W(n)_{0}\cong \mathfrak{gl}(n),\;S(n)_{0}=\widetilde{S}(n)_{0}\cong \mathfrak{sl}(n),\; H(n)_{0}\cong \mathfrak{so}(n).$$
Note that the $-1$-degree components of these superalgebras is the dual module of natural module of the $0$-degree components.

Let $L=\oplus_{i\in \mathbb{Z}}L_{i}$ be a $\mathbb{Z}$-graded Lie superalgbra, $H_{L}$ be the standard Cartan subalgebra  of $L$, $\theta\in H_{L}^{*}$ be the zero weight, $\Delta^{L}$ be the root system of $L$ and $\Delta^{L}_{i}$ be the root system of $L_{i}$.
If $x\in L_{\varepsilon}$ (resp. $x\in L_{i}$), where $\varepsilon\in  H_{L}^{*}$ and $i\in \mathbb{Z}$, then we say that $x$ is $H_{L}^{*}$-homogeneous (resp. $\mathbb{Z}$-homogeneous) of degree $\varepsilon$ (resp. $i$) and we write $||x||=\varepsilon$ (resp. $\mathrm{deg}(x)=i$).

Let us describe the roots  of Cartan type Lie superalgebras. If  $L=W(n),$ we choose the standard basis $\{\varepsilon_{1}, \ldots, \varepsilon_{n}\}$ in $H_{L}^{*},$ and then
$$
\Delta^{L}=\{\varepsilon_{i_{1}}+\cdots +\varepsilon_{i_{k}}-\varepsilon_{i} \mid  1\leq i_{1} <\cdots< i_{k}\leq n,  1\leq i\leq n \}.
$$
The root systems of $S(n)$ and $\widetilde{S}(n)$ are obtained from the root system of $W(n)$
by removing the roots $\varepsilon_{1}+\cdots +\varepsilon_{n}-\varepsilon_{i},$ where $1\leq i \leq n.$
Finally if $L=H(n)$, then
$$\Delta^{L}=\left\{\pm \varepsilon_{i_{1}}\pm\cdots \pm\varepsilon_{i_{k}}  \mid  1\leq i_{1} <\cdots< i_{k}\leq \left[\frac{n}{2}\right]\right\}.$$

\section{General lemmas}
In this section, let us establish  several  lemmas, which will be used in  proving the main result for this paper. Put $\mathcal{C}=\sum_{i=1}^{n}x_{i}\partial_{i}$ and $\widetilde{H}(n)=\left\{\mathrm{D_{H}}(f)\mid  f\in \Lambda(n)\right\} $. By \cite{s}, we have the following Lemma.
\begin{lemma}\label{l1}
Let $L$ be a Cartan type Lie superalgebra. Then  $\mathrm{Der}L=\mathrm{ad}L',$ where
\[L'\cong\left\{
 \begin{array}{lll}
L, &\mbox{if  $L=W(n), \widetilde{S}(n),$ }\\
L\oplus \mathbb{C}\mathcal{C}, &\mbox{if  $L=S(n),$ }\\
\widetilde{H}(n)\oplus \mathbb{C}\mathcal{C}, &\mbox{if  $L=H(n).$ }
\end {array}
\right. \]
\end{lemma}
By Lemma \ref{l1} and a simple computation, we  shows
$\Delta^{L'}=\Delta^{L}$ for   $L=W(n), S(n), \widetilde{S}(n)$ or $H(n).$
\begin{lemma}\label{chenl1}
Let $L=H(n),$ $-1\leq k, l\leq \xi_{H}$ and $x\in L_{k}$. If $k+l=\xi_{H}$ and $[x, L_{l}]=0,$ then  $x=0.$
\end{lemma}
\begin{proof}
We write
$$x=\sum_{1\leq i_{1}<\cdots< i_{k+2}\leq n}a_{i_{1}\cdots i_{k+2}}\mathrm{D_{H}}(x_{i_{1}}\cdots x_{i_{k+2}})\in L_{k},$$
where $a_{i_{1}\cdots i_{k+2}}\in \mathbb{C}.$
For $1\leq i_{1}<\cdots< i_{k+2}\leq n,$  put
$$I=\{i_{1},\ldots, i_{k+2}\}, J=\{j_{1}, \ldots, j_{l+1}\}=\{1, \ldots, n\}\backslash I.$$
 Our discussion is divided
into two cases.

$\emph{Case 1.}$ If $i_{1}'\notin J,$ then $0\neq \mathrm{D_{H}}(x_{i_{1}'}x_{j_{1}}\cdots x_{j_{l+1}})\in L_{l}.$ So
 \begin{eqnarray}\label{1}
 \left[\sum_{1\leq i_{1}<\cdots< i_{k+2}\leq n}a_{i_{1}\cdots i_{k+2}}\mathrm{D_{H}}(x_{i_{1}}\cdots x_{i_{k+2}}), \mathrm{D_{H}}(x_{i_{1}'}x_{j_{1}}\cdots x_{j_{l+1}})\right]=0.
  \end{eqnarray}
  By equation (\ref{hee1.2}) and observing the coefficients of terms   $\mathrm{D_{H}}(x_{i_{2}}\cdots x_{i_{k+2}}x_{j_{1}}\cdots x_{j_{l+1}})$ in equation (\ref{1}),  we can obtain  $a_{i_{1}\cdots i_{k+2}}=0.$

$\emph{Case 2.}$ If $i_{1}'\in J,$ without loss of generality we can let $i_{1}'=j_{1},$  then $0\neq \mathrm{D_{H}}(x_{i_{1}}x_{j_{1}}\cdots x_{j_{l+1}}) \in L_{l}.$ So
  \begin{eqnarray}\label{2}
 \left[\sum_{1\leq i_{1}<\cdots< i_{k+2}\leq n}a_{i_{1}\cdots i_{k+2}}\mathrm{D_{H}}(x_{i_{1}}\cdots x_{i_{k+2}}), \mathrm{D_{H}}(x_{i_{1}}x_{j_{1}}\cdots x_{j_{l+1}})\right]=0.
  \end{eqnarray}
By equation (\ref{hee1.2}) and observing the coefficients of terms   $\mathrm{D_{H}}(x_{i_{1}}\cdots x_{i_{k+2}}x_{j_{2}}\cdots x_{j_{l+1}})$  in equation (\ref{2}) we can obtain   $a_{i_{1}\cdots i_{k+2}}=0.$

In summary, we have $x=0.$
\end{proof}
For a nonempty subset
$V$ of a Lie superalgebra $L$, we write $\mathrm{alg}(V)$ for subalgebra of $L$, which  generated by $V$. By  a direct calculation, we have the following lemma.
\begin{lemma}\label{l2}
Let $L$ be a  Cartan type Lie superalgebra. Then  the following properties hold:
\begin{itemize}
  \item [(1)] $[L_{-1}, L_{i}]=L_{i-1}$.
  \item [(2)] $\mathrm{alg}(L_{-1}\oplus L_{0}\oplus L_{1})=L$.
  \item [(3)] $L'$ is transitive, that is $a\in \oplus_{i\geq 0}L'_{i} $ and  $\left[a, L'_{-1}\right]=0,$ then $a=0.$
  \item [(4)] $L$ is irreducible, that is $L_{-1}$ is irreducible $L_{0}$-module.
\end{itemize}
\end{lemma}
 Suppose $L$ is a Cartan type Lie superalgebra.
Since $L$   has the weight space decompositions and $\mathbb{Z}$-grading,   of course $\mathrm{BDer}(L)$ inherits a $H_{L}^{*}$-grading and $\mathbb{Z}$-grading from $L$ as above, that is
$$
\mathrm{BDer}(L)=\bigoplus_{\varepsilon\in H_{L}^{*}}\mathrm{BDer}(L)_{\varepsilon}, \mathrm{BDer}(L)=\bigoplus_{i\in \mathbb{Z}}\mathrm{BDer}(L)_{i},
$$
where
$$
\mathrm{BDer}(L)_{\varepsilon}=\{f\in \mathrm{BDer}(L)\mid f(L_{\eta},L_{\zeta})\subseteq L_{\varepsilon+\eta+\zeta},  \forall \eta, \zeta\in H_{L}^{*}\},
$$
$$
\mathrm{BDer}(L)_{i}=\{f\in \mathrm{BDer}(L)\mid f(L_{j},L_{k})\subseteq L_{i+j+k},  \forall j,k\in  \mathbb{Z}\}.
$$
As the proof of the Lemma 3.3 in \cite[ p97]{yt}, we can prove  the  following Lemma, which  is useful to investigate the super-biderivations of Cartan type Lie superalgebra.
\begin{lemma}\label{l55}
Let $L$ be a  Cartan type Lie superalgebra and $f\in  \mathrm{BDer}(L)$. Then there are linear maps $\phi_{f}$ and $\psi_{f}$ from $L$ into $L'$  such that
\begin{eqnarray*}\label{e2}
f(x, y)=[\phi_{f}(x), y]=[x, \psi_{f}(y)]
\end{eqnarray*}
for all $x, y\in L.$
Furthermore, we have the following properties hold:
\begin{itemize}
  \item [(1)] If $f\in  \mathrm{BDer}(L)_{\varepsilon}$, then  $||\phi_{f}||=||\psi_{f}||=\varepsilon$.
  \item [(2)] If $f\in  \mathrm{BDer}(L)_{i}$, then $\mathrm{deg}(\phi_{f})=\mathrm{deg}(\psi_{f})=i$.
\end{itemize}

\end{lemma}
By Lemmas \ref{l2} and  \ref{l55}, we have the following  two lemmas.
\begin{lemma}\label{l56}
Let $L$ be a   Cartan type Lie superalgebra and $f\in  \mathrm{BDer}(L).$ If  $\phi_{f}(L_{-1}\oplus L_{\xi_{L}})=0$ or $\psi_{f}(L_{-1}\oplus L_{\xi_{L}})=0,$ then  $f$ is zero.
\end{lemma}
\begin{proof}
We only consider the case   $\phi_{f}(L_{-1}\oplus L_{\xi_{L}})=0$. The other case may be treated in a similar way. By Lemma \ref{l55}, we have $f(x,y)=[\phi_{f}(x),y]=0$ for all $x\in L_{-1}$ and $y\in L.$ Then by Lemma \ref{l2} (1) and (\ref{e1}) we can prove the  following fact:
If  $f(x,y)=0$ for all  $x\in L_{i}$ and $y\in L,$ then $f(z,y)=0$ for all  $z\in L_{i-1}$ and $y\in L.$
By Lemma \ref{l55}, we have $f(x,y)=[\phi_{f}(x),y]=0$ for all $x\in  L_{\xi_{L}}$ and $y\in L.$ This together with the above fact implies that  $f$ is zero.
\end{proof}
\begin{lemma}\label{l57}
Let $L$ be a   Cartan type Lie superalgebra and $f\in  \mathrm{BDer}(L)$. If  $\phi_{f}(L_{-1}\oplus L_{0}\oplus L_{1})=0$ or $\psi_{f}(L_{-1}\oplus L_{0}\oplus L_{1})=0,$ then  $f$ is zero.
\end{lemma}
\begin{proof}
We only consider the case    $\phi_{f}(L_{-1}\oplus L_{0}\oplus L_{1})=0$. The other case may be treated in a similar way. Then
by Lemma \ref{l55}, we have $f(x,y)=0$ for all $x\in L_{-1}\oplus L_{0}\oplus L_{1}$ and $y\in L.$
  By Lemma \ref{l2}  (2) and (\ref{e1}), we have $f(z,y)=0$ for all  $z, y\in L.$ Then $f$ is zero.
\end{proof}
The following Lemma is given by Theorem 2.4 in \cite{t1}.
\begin{lemma}\label{yuanl3}
Let $L$ be a finite-dimensional simple Lie algebra over  $\mathbb{C}$. Then every biderivation of $L$ is inner.
\end{lemma}

\section{Super-biderivations  of $W(n), S(n)$ and  $\widetilde{S}(n)$}
In this section we shall
characterize the super-biderivations Lie superalgebras of  $W(n), S(n)$ and  $\widetilde{S}(n)$. Let $L=W(n), S(n)$ or  $\widetilde{S}(n)$. For $\varepsilon=\sum_{i=1}^{n}k_{i}\varepsilon_{i}\in H_{L}^{*}$, we write
$l(\varepsilon)=\sum_{i=1}^{n}k_{i}.$ Denote
$$A+B=\{\alpha+\beta\mid \alpha\in A, \beta\in B\},$$
where $A, B\subseteq H_{L}^{*}$.
 Then we have the following  super-biderivations vanishing  Proposition.
\begin{proposition}\label{pro4.1}
Let $L=W(n), S(n)$ or $\widetilde{S}(n)$   and $f\in \mathrm{BDer}(L)_{\varepsilon}$, where  $  \varepsilon\in H_{L}^{*}.$ If $\varepsilon\neq 0,$ then   $f$ is zero.
\end{proposition}
\begin{proof}
Let $\varepsilon=\sum_{i=1}^{n}k_{i}\varepsilon_{i},$ where $k_{i}\in \mathbb{C}.$ Our discussion is divided
into five cases.

$\emph{Case 1.}$ Suppose   $k_{i}\notin \mathbb{Z}$  for some $1\leq i\leq n.$ In this case, we have $\left(\varepsilon+\Delta^{L}\right) \cap\Delta^{L'} =\emptyset.$  Since $||\phi_{f}(x)||\in \varepsilon+\Delta^{L}$ for all $x\in L,$ then $\phi_{f}(x)=0.$  By Lemma \ref{l55}, we can obtain $f$ is zero.

$\emph{Case 2.}$ Suppose $k_{i}\in \mathbb{Z}$ for all $1\leq i\leq n$  and  $k_{j}< 0$  for some $1\leq j\leq n.$  In this case, we have $\varepsilon-\varepsilon_{j}\notin \Delta^{L}$ for all $1\leq j \leq n.$ Note that $\phi_{f}(\partial_{j}), \psi_{f}(\partial_{j})\in L'_{\varepsilon-\varepsilon_{j}}.$ Then $\phi_{f}(\partial_{j})=\psi_{f}(\partial_{j})=0$ for all $1\leq j \leq n.$ Next, Our discussion is divided
into two subcases.

$\emph{Subcase 2.1.}$ Suppose $\varepsilon\notin -\Delta^{L}_{\xi_{L}}+\Delta^{L}_{-1}.$ In this case, we have $\left(\varepsilon+\Delta^{L}_{\xi_{L}}\right)\cap\Delta^{L'}_{-1}=\emptyset.$ Since  $$||\phi_{f}(L_{\xi_{L}})||, ||\psi_{f}(L_{\xi_{L}})||\in \varepsilon+\Delta^{L}_{\xi_{L}},$$ then $||\phi_{f}(L_{\xi_{L}})||, ||\psi_{f}(L_{\xi_{L}})||\notin\Delta^{L'}_{-1}$. Therefore $\phi_{f}(L_{\xi_{L}}), \psi_{f}(L_{\xi_{L}})\in \oplus_{i\geq 0}L'_{i}.$  By Lemma \ref{l55}, we have
      $$
0=[\phi_{f}(\partial_{j}), L_{\xi_{L}}]=[\partial_{j}, \psi_{f}( L_{\xi_{L}})]
$$
for all $1\leq j \leq n.$ By Lemma \ref{l2} (3), we have $\psi_{f}( L_{\xi_{L}})=0$.  By Lemma \ref{l56}, we have $f$ is zero.

 $\emph{Subcase 2.2.}$ Suppose $\varepsilon\in -\Delta_{\xi_{L}}^{L}+\Delta_{-1}^{L}.$ By direct calculation, we have
    \[
-\Delta_{\xi_{L}}^{L}+\Delta_{-1}^{L}=\left\{
 \begin{array}{ll}
\left\{-\sum_{i=1}^{n}\varepsilon_{i}+\varepsilon_{j}-\varepsilon_{l}\mid 1\leq j, l\leq n\right\},&\mbox{if  $L=W(n),$ }\\
\left\{-\sum_{j\neq i=1}^{n}\varepsilon_{i}+\varepsilon_{k}-\varepsilon_{l}\mid 1\leq j, k, l\leq n\right\},&\mbox{if $L=S(n), \widetilde{S}(n)$.}
\end {array}
\right.
\]
   Then for $\varepsilon\in -\Delta_{\xi_{L}}^{L}+\Delta_{-1}^{L},$ we have
\[
l(\varepsilon)=\left\{
 \begin{array}{ll}
-n,&\mbox{if  $L=W(n),$ }\\
-n+1,&\mbox{if $L=S(n), \widetilde{S}(n)$.}
\end {array}
\right.
\]
Sequentially, we have
\[
l(||\phi_{f}(x)||)=l(\varepsilon)+l(||x||)=\left\{
 \begin{array}{ll}
-n,&\mbox{if  $L=W(n),$ }\\
-n+1,&\mbox{if $L=S(n), \widetilde{S}(n)$}
\end {array}
\right.
\]
for all $x\in L_{0}.$ Note that $l(\alpha)\geq-1$ for all $\alpha\in \Delta^{L'}$ and $n\geq 4.$ Therefore, $\phi_{f}(x)=0,$ that is
    $\phi_{f}(L_{0})=0.$  Similarly, we can prove that $\phi_{f}(L_{1})=0.$ By Lemma \ref{l57}, we can obtain
      $f$ is zero.

$\emph{Case 3.}$ Suppose   $k_{i}\in \mathbb{Z}$ for all $1\leq i\leq n$ and  $k_{j}\geq 3$  for some $1\leq j\leq n.$
Note that for any $\alpha=\sum_{i=1}^{n}l_{i}\varepsilon_{i}\in \Delta^{L},$ we have $l_{j}\in \{-1, 0, 1\}.$ Then $k_{j}+l_{j}\geq 2.$ So
$\varepsilon+\alpha\notin \Delta^{L'},$
that is  $(\varepsilon+\Delta^{L})\cap \Delta^{L'}=\emptyset.$ Since $||\phi_{f}(x)||\in \varepsilon+\Delta^{L}$ for all $x\in L,$ then $\phi_{f}(x)=0.$  By Lemma \ref{l55}, we can obtain $f$ is zero.

$\emph{Case 4.}$ Suppose $k_{i}\in \mathbb{Z}$ for all $1\leq i\leq n$ and  $k_{p}=k_{q}= 2$  for some $1\leq p\neq q\leq n.$  Note that for any $\alpha=\sum_{i=1}^{n}l_{i}\varepsilon_{i}\in \Delta^{L},$ we have $l_{p}, l_{q}\in \{-1, 0, 1\}$, but $l_{p}$ and $l_{q}$ are not all equal to -1 at the same time. Then $k_{p}+l_{p}\geq 2$  or $k_{q}+l_{q}\geq 2.$ So
$\varepsilon+\alpha\notin \Delta^{L'},$
that is  $(\varepsilon+\Delta^{L})\cap \Delta^{L'}=\emptyset.$ Since $||\phi_{f}(x)||\in \varepsilon+\Delta^{L}$ for all $x\in L,$ then $\phi_{f}(x)=0.$  By Lemma \ref{l55}, we can obtain $f$ is zero.

$\emph{Case 5.}$ Suppose $\varepsilon=\sum_{i=1}^{n}k_{i}\varepsilon_{i}$, where  $k_{i}\in \{0, 1, 2\}$  and at  most one of  $k_{i}$ is equal to 2. Without loss of generality, we can write  $\varepsilon=a\varepsilon_{1}+\varepsilon_{2}+\cdots+\varepsilon_{i},$ where $1\leq i \leq n$ and $a\in \{1, 2\}.$
Since $\varepsilon\neq \theta,$ $l(\varepsilon)>0.$
Then
\[
l(\varepsilon+\alpha)> \left\{
 \begin{array}{ll}
n-1,&\mbox{if  $L=W(n),$ }\\
n-2,&\mbox{if $L=S(n), \widetilde{S}(n)$}
\end {array}
\right.
\]
for all $\alpha\in \Delta^{L}_{\xi_{L}}.$
Note that
\[
l(\beta)\leq \left\{
 \begin{array}{ll}
n-1,&\mbox{if  $L=W(n),$ }\\
n-2,&\mbox{if $L=S(n), \widetilde{S}(n)$}
\end {array}
\right.
\]
for all $\beta\in \Delta^{L'}.$
Therefore, $\left(\varepsilon+\Delta^{L}_{\xi_{L}}\right)\cap \Delta^{L'}=\emptyset.$ Then $\phi_{f}\left(L_{\xi_{L}}\right)=\psi_{f}(L_{\xi_{L}})=0.$
Next, Our discussion is divided
into two  subcases.

 $\emph{Subcase 5.1.}$ Suppose $\varepsilon=2\varepsilon_{1}+\varepsilon_{2}+\cdots+\varepsilon_{i}.$  In this case, we have $\phi_{f}(H_{L})=\phi_{f}(H_{L})=0.$ By Lemma \ref{l55}, we have
\begin{eqnarray}\label{yc1}\nonumber
0&=&\left[\phi_{f}(h), \partial_{l}-\delta_{L,\widetilde{S}}x_{1}x_{2}\cdots x_{n}\partial_{l}\right]=\left[h, \psi_{f}(\partial_{l}-\delta_{L,\widetilde{S}}x_{1}x_{2}\cdots x_{n}\partial_{l})\right]\\
&=&(2\varepsilon_{1}+\varepsilon_{2}+\cdots+\varepsilon_{i}-\varepsilon_{l})(h)\psi_{f}\left(\partial_{l}-\delta_{L,\widetilde{S}}x_{1}x_{2}\cdots x_{n}\partial_{l}\right)
\end{eqnarray}
      for all $h\in H_{L}.$ Since $2\varepsilon_{1}+\varepsilon_{2}+\cdots+\varepsilon_{i}-\varepsilon_{l}\neq \theta,$ there exists $h\in H_{L}$ such that
      $(2\varepsilon_{1}+\varepsilon_{2}+\cdots+\varepsilon_{i}-\varepsilon_{l})(h)\neq 0.$ By (\ref{yc1}), we have $\psi_{f}\left(\partial_{l}-\delta_{L,\widetilde{S}}x_{1}x_{2}\cdots x_{n}\partial_{l}\right)=0,$ that is $\psi_{f}(L_{-1})=0.$  By Lemma \ref{l56}, we have $f$ is zero.

  $\emph{Subcase 5.2.}$ Suppose $\varepsilon=\varepsilon_{1}+\varepsilon_{2}+\cdots+\varepsilon_{i}.$  In this case, we have $$\phi_{f}(x_{1}\partial_{j}), \psi_{f}(x_{1}\partial_{j})\in L'_{2\varepsilon_{1}+\varepsilon_{2}+\cdots+\varepsilon_{i}-\varepsilon_{j}}=0$$
      for all $j\neq 1.$ Since $\phi_{f}\left(\partial_{1}-\delta_{L,\widetilde{S}}x_{1}x_{2}\cdots x_{n}\partial_{1}\right)\in L'_{\varepsilon_{2}+\cdots+\varepsilon_{i}}$ and $$L'_{\varepsilon_{2}+\cdots+\varepsilon_{i}}=\mathrm{span}_{\mathbb{C}}\{x_{2}\cdots x_{i}x_{j}\partial_{j}\mid j=1, i+1,\ldots, n\}\cap L',$$ we can write
      $$\phi_{f}\left(\partial_{1}-\delta_{L,\widetilde{S}}x_{1}x_{2}\cdots x_{n}\partial_{1}\right)=a_{1}x_{2}\cdots x_{i}x_{1}\partial_{1}+\sum_{j=i+1}^{n}a_{j}x_{2}\cdots x_{i}x_{j}\partial_{j},$$
       where $a_{1}, a_{j}\in \mathbb{C}.$
        By Lemma \ref{l55}, we have
\begin{eqnarray*}\label{yt2}
0&=&\left[\partial_{1}\delta_{L,\widetilde{S}}x_{1}x_{2}\cdots x_{n}\partial_{1},\psi_{f}(x_{1}\partial_{2})\right]=\left[\phi_{f}(\partial_{1}-\delta_{L,\widetilde{S}}x_{1}x_{2}\cdots x_{n}\partial_{1}), x_{1}\partial_{2}\right]\\
&=&\left[a_{1}x_{2}\cdots x_{i}x_{1}\partial_{1}+\sum_{j=i+1}^{n}a_{j}x_{2}\cdots x_{i}x_{j}\partial_{j},x_{1}\partial_{2}\right]\\
&=&a_{1}x_{2}\cdots x_{i}x_{1}\partial_{2}-\sum_{j=i+1}^{n}a_{j}x_{1}x_{3}\cdots x_{i}x_{j}\partial_{j}.
\end{eqnarray*}
 Then $a_{1}=a_{i+1}=\cdots=a_{n}=0,$ that is $\phi_{f}\left(\partial_{1}-\delta_{L,\widetilde{S}}x_{1}x_{2}\cdots x_{n}\partial_{1}\right)=0.$  Similarly, we can prove that $\psi_{f}\left(\partial_{1}-\delta_{L,\widetilde{S}}x_{1}x_{2}\cdots x_{n}\partial_{1}\right)=0.$ Let $h\in H_{L}.$
 Since $\phi_{f}(h)\in L'_{\varepsilon_{1}+\cdots+\varepsilon_{i}}$ and $$L'_{\varepsilon_{1}+\cdots+\varepsilon_{i}}=\mathrm{span}_{\mathbb{C}}\{x_{1}\cdots x_{i}x_{j}\partial_{j}\mid j= i+1,\ldots, n\}\cap L',$$ we can write
      $$\phi_{f}(h)=\sum_{j=i+1}^{n}a_{j}x_{2}\cdots x_{i}x_{j}\partial_{j},$$
       where $ a_{j}\in \mathbb{C}.$
        By Lemma \ref{l55}, we have
\begin{eqnarray*}\label{yt2}
0&=&\left[\phi_{f}(\partial_{1}-\delta_{L,\widetilde{S}}x_{1}x_{2}\cdots x_{n}\partial_{1}),h\right]=\left[\partial_{1}-\delta_{L,\widetilde{S}}x_{1}x_{2}\cdots x_{n}\partial_{1}, \psi_{f}(h)\right]\\
&=&\left[\partial_{1}-\delta_{L,\widetilde{S}}x_{1}x_{2}\cdots x_{n}\partial_{l}, \sum_{j=i+1}^{n}a_{j}x_{1}\cdots x_{i}x_{j}\partial_{j}\right]\\
&=&\sum_{j=i+1}^{n}a_{j}x_{2}x_{3}\cdots x_{i}x_{j}\partial_{j}.
\end{eqnarray*}
 Then $a_{i+1}=\cdots=a_{n}=0,$ that is $\psi_{f}(h)=0.$  Similarly, we can prove that $\phi_{f}(h)=0.$
 Note that $\phi_{f}(\partial_{l}), \psi_{f}(\partial_{l})\in L'_{\varepsilon_{1}+\cdots+\varepsilon_{i}-\varepsilon_{l}}$ and there exists $h\in H_{L}$ such that
      $(\varepsilon_{1}+\cdots+\varepsilon_{i}-\varepsilon_{l})(h)\neq 0$  for all $l\neq 1.$
 By Lemma \ref{l55}, we have
      \begin{eqnarray*}\label{yt2}
0&=&\left[\partial_{l}-\delta_{L,\widetilde{S}}x_{1}x_{2}\cdots x_{n}\partial_{l}, \psi_{f}(h)\right]=\left[\phi_{f}(\partial_{l}-\delta_{L,\widetilde{S}}x_{1}x_{2}\cdots x_{n}\partial_{l}),h\right]\\
&=&-(\varepsilon_{1}+\cdots+\varepsilon_{i}-\varepsilon_{l})(h)\phi_{f}\left(\partial_{l}-\delta_{L,\widetilde{S}}x_{1}x_{2}\cdots x_{n}\partial_{l}\right).
\end{eqnarray*}
 Then $\phi_{f}\left(\partial_{l}-\delta_{L,\widetilde{S}}x_{1}x_{2}\cdots x_{n}\partial_{l}\right)=0$ for all $l\neq 1.$ In summary, $\phi(L_{-1})=0.$ By Lemma \ref{l56}, we have $f=0.$

\end{proof}
\begin{proposition}\label{pro4.2}
Let $L=W(n), S(n)$ or $\widetilde{S}(n)$. If  $f\in \mathrm{BDer}(L)_{\theta}$, then   $f$ is inner, that is there is $\lambda\in \mathbb{C}$ such that $f(x, y)=\lambda[x, y]$ for all $x, y\in L.$
\end{proposition}
\begin{proof}
Let $1\leq i\leq n$ and $\varepsilon\in \Delta_{\xi_{L}}^{L'}.$ Since $\dim L'_{-\varepsilon_{i}}=1$ and  $\dim \left(L'_{\xi_{L}}\right)_{\varepsilon}=1,$  then we have
$$\phi_{f}\left(\partial_{i}-\delta_{L,\widetilde{S}}x_{1}x_{2}\cdots x_{n}\partial_{i}\right)=\lambda_{i}\left(\partial_{i}-\delta_{L,\widetilde{S}}x_{1}x_{2}\cdots x_{n}\partial_{i}\right),$$
$$\psi_{f}\left(\partial_{i}-\delta_{L,\widetilde{S}}x_{1}x_{2}\cdots x_{n}\partial_{i}\right)=\beta_{i}\left(\partial_{i}-\delta_{L,\widetilde{S}}x_{1}x_{2}\cdots x_{n}\partial_{i}\right),$$
$\phi_{f}(x)=\lambda_{x}x$ and $\psi_{f}(x)=\beta_{x}x,$ where $\lambda_{i}, \beta_{i}, \lambda_{x}, \beta_{x}\in \mathbb{C}$ and  $x\in \left(L'_{\xi_{L}}\right)_{\varepsilon}.$
By Lemma \ref{l55}, we have
$$
\left[\phi_{f}(\partial_{i}-\delta_{L,\widetilde{S}}x_{1}x_{2}\cdots x_{n}\partial_{i}), x\right]=\left[\partial_{i}-\delta_{L,\widetilde{S}}x_{1}x_{2}\cdots x_{n}\partial_{i}, \psi_{f}(x)\right].
$$
Then  $\lambda_{i}=\beta_{x}.$ Similarly, we have $\beta_{i}=\lambda_{x}.$   Therefore, $\lambda_{1}=\cdots=\lambda_{n}$  and $\beta_{1}=\cdots=\beta_{n},$  denoted $\lambda$ and $\beta$, respectively.  We   put
$\lambda=\lambda_{1} $ and $\beta=\beta_{1}.$
Let $y\in \sum_{j\geq 0}L_{j}$ and $1\leq i \leq n$,
then
\begin{eqnarray*}
f\left(\partial_{i}-\delta_{L,\widetilde{S}}x_{1}x_{2}\cdots x_{n}\partial_{i},y\right)&=&\left[\phi_{f}(\partial_{i}-\delta_{L,\widetilde{S}}x_{1}x_{2}\cdots x_{n}\partial_{i}),y\right]\\
&=&\lambda\left[\partial_{i}-\delta_{L,\widetilde{S}}x_{1}x_{2}\cdots x_{n}\partial_{i},y\right]\\
&=&\left[\partial_{i}-\delta_{L,\widetilde{S}}x_{1}x_{2}\cdots x_{n}\partial_{i},\psi_{f}(y)\right].
\end{eqnarray*}
Then
$$\left[\partial_{i}-\delta_{L,\widetilde{S}}x_{1}x_{2}\cdots x_{n}\partial_{i},\lambda y-\psi_{f}(y)\right]=0$$
 for all $1\leq i \leq n.$ Note that $ \psi_{f}(\oplus_{j\geq 0}L_{j})\subseteq \oplus_{j\geq 0}L'_{j},$ So by Lemma \ref{l2} (3), we have $\psi_{f}(y)=\lambda y.$
Similarly, we have $\phi_{f}(y)=\beta y.$
By Lemma \ref{l55}, we have
$$f(x_{1}\partial_{2},x_{2}\partial_{1})=[\phi_{f}(x_{1}\partial_{2}),x_{2}\partial_{1}]=\beta [x_{1}\partial_{2},x_{2}\partial_{1}]=[x_{1}\partial_{2},\psi_{f}(x_{2}\partial_{1})]=\lambda [x_{1}\partial_{2},x_{2}\partial_{1}].$$
Then $\lambda =\beta.$ In summary,   $\phi_{f}(x)=\lambda x$ for all $x\in L.$ That is $f$ is inner.
\end{proof}
By Propositions \ref{pro4.1} and \ref{pro4.2}, we have the following theorem.
\begin{theorem}
Let $L=W(n), S(n)$ or $\widetilde{S}(n)$. Then
$$\mathrm{BDer}(L)=\mathrm{IBDer}(L).$$
\end{theorem}

\section{Super-biderivations  of $H(n)$}
In this section we shall
characterize the super-biderivations   of   $H(n)$.  We have the following  super-biderivations vanishing  Proposition.
\begin{proposition}\label{pro1}
Let $L=H(n)$   and $f\in \mathrm{BDer}(L)_{i}$, where  $  i\in \mathbb{Z}.$ If $i\neq 0,$ then   $f$ is zero.
\end{proposition}
\begin{proof} Our discussion is divided
into three cases.

$\emph{Case 1.}$ Suppose $i< -\xi_{H}$. Since $\xi_{H}=n-2$ and $n>4$,      $\phi_{f}(L_{-1}\oplus L_{0}\oplus L_{1})=0$. By Lemma \ref{l57}, we have $f=0.$

$\emph{Case 2.}$ Suppose $-\xi_{H}\leq i\leq -1$. In this case, we have $\phi_{f}(L_{-1})=\psi_{f}(L_{-1})=0$ and  $\psi_{f}(L_{\xi_{H}})\subseteq \oplus_{j\geq 0}L'_{j}$. Let $x\in L_{\xi_{H}}$ and $1\leq j\leq n.$ Then
by Lemma \ref{l55}, we have
$$0=[\phi_{f}(\partial_{j}),x]=f(\partial_{j},x)=[\partial_{j},\psi_{f}(x)].$$
By Lemma \ref{l2} (3), we have $\psi_{f}(x)=0$, that is  $\psi_{f}(L_{\xi_{H}})=0.$ By Lemma \ref{l56}, we have $f=0.$

$\emph{Case 3.}$ Suppose $i\geq 1.$ In this case, we have
$$\phi_{f}(L_{\xi_{H}-i+1}\oplus\cdots\oplus L_{\xi_{H}})=\psi_{f}(L_{\xi_{H}-i+1}\oplus\cdots\oplus L_{\xi_{H}})=0.$$
 Let $x\in L_{\xi_{H}-i+1}$ and $1\leq j \leq n.$
 By Lemma \ref{l55}, we have
$$0=[\partial_{j},\psi_{f}(x)]=f(\partial_{j},x)=[\phi_{f}(\partial_{j}),x],$$
that is $[\phi_{f}(\partial_{j}), L_{\xi_{H}-i+1}]=0.$ Note that $\phi_{f}(\partial_{j})\in L'_{i-1}$.
By Lemma \ref{chenl1}, we have $\phi_{f}(\partial_{j})=0.$ That is $\phi_{f}(L_{-1})=0.$  By Lemma \ref{l56}, we have $f$ is zero.
\end{proof}
\begin{proposition}\label{pro2}
Let $L=H(n)$.    If  $f\in \mathrm{BDer}(L)_{0}$, then   $f$ is inner, that is there is $\lambda\in \mathbb{C}$ such that $f(x, y)=\lambda[x, y]$ for all $x, y\in L.$
\end{proposition}
\begin{proof}
In this case, we have $f\mid_{L_{0}\times L_{0}}$ is super-biderivation of $L_{0}$. By Lemma \ref{yuanl3}, we can write $\phi_{f}(x)=\psi_{f}(x)=\lambda x$ for all $x\in L_{0}.$ By Lemma \ref{l55}, we have
$$\lambda[x, y]=[\phi_{f}(y),y]=f(x,y)=[x,\psi_{f}(y)]$$
for all $x\in L_{0}$ and $y\in L_{-1}$. Then $[L_{0}, \lambda y-\psi_{f}(y)]=0.$ By Lemma \ref{l2} (4), we have $\psi_{f}(y)=\lambda y$ for all  $y\in L_{-1}$.
By Lemma \ref{l55}, we have
$$\lambda[x, y]=[\phi_{f}(x),y]=f(x,y)=[x,\psi_{f}(y)]$$
for all $x\in L_{-1}$ and $y\in \oplus_{i\geq 1}L_{i}$. Then $[L_{-1}, \lambda y-\psi_{f}(y)]=0.$ By Lemma \ref{l2} (3), we have $\psi_{f}(y)=\lambda y.$
In summary,   $\psi_{f}(x)=\lambda x$ for all $x\in L.$ That is $f$ is inner.
\end{proof}
By Propositions \ref{pro1} and \ref{pro2}, we have the following theorem.
\begin{theorem}
Let $L=H(n)$. Then
$$\mathrm{BDer}(L)=\mathrm{IBDer}(L).$$
\end{theorem}


\end{document}